\documentclass{pubNEW}

\authornames{L. Sacerdote, O. Telve and C. Zucca} 
\shorttitle{Joint densities of FPT's through two boundaries} 

\usepackage{graphicx}
\usepackage{setspace}

\begin{document}
\doublespacing 

\title{Joint densities of first hitting times of a \\diffusion process through two time dependent boundaries}  

\authorone[Department of Mathematics \lq\lq G. Peano\rq\rq, University of Torino]{\\Laura Sacerdote\\ Ottavia Telve\\ Cristina Zucca} %
\addressone{Department of Mathematics \lq\lq G. Peano\rq\rq, University of Torino, Via Carlo Alberto 10, 10123 Torino, Italy} 

\begin{abstract}
Consider a one dimensional diffusion process on the diffusion interval $I$ originated in $x_0\in I$. Let $a(t)$ and $b(t)$ be two continuous functions of $t$, $t>t_0$ with bounded derivatives and with $a(t)<b(t)$ and $a(t),b(t)\in I$, $\forall t>t_0$. 
We study the joint distribution of the two
random variables $T_a$ and $T_b$, first hitting times of the diffusion process through the two boundaries $a(t)$ and $b(t)$, respectively. 
We express the joint distribution of $T_a, T_b$ in terms of $P(T_a<t,T_a<T_b)$ and $P(T_b<t,T_a>T_b)$ and we determine a system of integral equations verified by these last probabilities. We propose a numerical algorithm to solve this system and we prove its convergence properties. Examples and modeling motivation for this study are also discussed.
\end{abstract}



\section{Introduction}\label{Sect:1}
Exit times of diffusion processes from a strip play an important role in a
variety of application ranging from computer science to engineering, from biology to metrology or finance (cf. \cite{Giorno,DL,NFK,PTZ,Sm}). According to the features of the model, constant or time dependent thresholds may bound the considered process. 
Typical examples are quality models with two tolerance bands. Some parameter may control exit times from the strip, with different effects on the exit time from the upper or the lower bound. The knowledge of the joint exit times pdf  clarifies the role of these parameters.
Another example is given by the survival probability of a population in a finite capacity environment or by tumor growth models (cf. \cite{Giorno}). Similar problems arise in metrology when we need to maintain the atomic clock error bounded by two tolerance bands. 
Moreover, avoiding an excessive increase of the error is of primary importance to improve GPS
instruments (cf. \cite{GTSZ}). In this setting the knowledge of the relationship between exit times from the upper and the lower boundary may suggest improvements to the clock reliability by acting on some parameters of the model involved in the joint distribution. 
Other possible applications can be found in finance where the interest focuses on the dependency between the times to sell or buy options when the level of gain or loss is preassigned. 
A large literature exists for the study of the first passage time of one dimensional diffusion processes through a boundary and analytical, numerical and simulation methods have been studied both for the direct (cf. \cite{BNR,BSZ,GS,GS1,GSZ,Pe-1}) and the inverse problem (cf. \cite{ZucSac}). However, most of these papers focuses on the one boundary problem, while for the two boundary case the few analytical results published rely either on the Brownian motion (cf. \cite{Ors}) or particular time dependent boundaries, corresponding to special symmetries, for specific diffusions (cf. \cite{BGNR,DGNR}). The existing results generally focus on the first exit time from the strip, while our interest lies in the joint distribution of the times when the process first attains the upper and the lower boundary, respectively. This paper aims to cover this subject considering the joint distribution between these times. Some results, presented in a recent paper \cite{GNR}, are related with those on the Laplace transforms presented in this paper. However their focus is not the joint distribution of exit times from a strip.

The notation and the existing results that will be used in this paper are introduced in Section 2, while Sections 3 and 4 are devoted to the presentation of our results. 
In Section 3 we determine the expression of the joint distribution of the exit times from the upper and the lower boundary.
The results are expressed in terms of first hitting time through a single boundary and of the probability of crossing the upper (lower) boundary for the first time at some instant preceding $t$ before crossing the lower (upper) boundary. Note that these probabilities are generally unknown. We prove then that they are the unique solution of a system of Volterra integral equations of the first kind. We also show that there exists an equivalent system of Volterra equations of the second type. When the boundaries are constant the Laplace transform method can be applied to solve the system, since the integrals of such system are of convolution type. Here we introduce three equivalent representations of the Laplace transform. In the case of the Brownian motion and constant boundaries a closed form expression for the joint distribution of the exit times from a strip is known (cf. \cite{Borod}).

In Section 4 we propose a numerical scheme for the solution of the system of integral equations and we determine the order of convergence. This method works for both constant and time dependent boundaries. In the case of two constant boundaries the Laplace transforms (cf. \cite{Abate}) of the probability of crossing the upper (lower) boundary for the first time at some instant preceding $t$ before crossing the lower (upper) boundary can be numerically inverted. Finally in Section 5, we present a set of examples.

\section{Mathematical Background and Notations}\label{Sect:2}

Let $X=\{X(t),t\geq t_0\}$ be a one-dimensional regular time homogeneous diffusion process defined on a suitable probability space $(\Omega, \mathcal{A}, \mathbb{P})$ such that $P(X(t_0)=x_0)=1$ and with diffusion interval $I$, where $I$ is an interval of the form $(r_1,r_2)$, $(r_1,r_2]$, $[r_1,r_2)$ or $[r_1,r_2]$ where $r_1=-\infty$ and/or $r_2=+\infty$ are admissible when the diffusion interval is open. If not specified, the diffusion interval is open and the endpoints $r_1$ and $r_2$ are natural boundaries.
Let
$F_{X(t)}(x|y,\tau)=P(X(t)\leq x|X(\tau)=y)$
be the transition probability distribution function (pDf) of the process $X$ and let $f_{X(t)}(x|y,\tau)$ be the corresponding transition probability density function (pdf).

Let $a(t)$ be a continuous functions with bounded derivatives. 
We denote as $T_a$ the first hitting time of the stochastic process $X$ across a boundary $a(t)\in I$
\begin{equation}\label{eq:2.5}
T_a=\inf\{t\geq t_0,X(t)=a(t)\}
\end{equation}
Its pDf is 
\begin{equation}
F_{T_a}(t|x_0,t_0)=P(T_a\leq t|X(t_0)=x_0)
\end{equation}
and $f_{T_a}(t|x_0,t_0)$ is the corresponding pdf. 
In the case of two boundaries $a(t)<b(t)$, $\forall t$, we indicate with $T_a$ and $T_b$ the first hitting times of the stochastic process $X$ across the boundaries $a(t)$ and $b(t)$ respectively. Aim of this paper is to study the dependency properties of $(T_a,T_b)$,
i.e. to determine  
\begin{equation}
F_{T_a,T_b}(t,s|x_0,t_0)=P(T_a\leq t, T_b\leq s|X(t_0)=x_0),
\end{equation}
the joint pDf of $(T_a,T_b)$ and the corresponding joint pdf $f_{T_a,T_b}(t,s|x_0,t_0)$. 

We define the following densities that distinguish the first boundary reached between the two ones delimiting the strip
\begin{equation}\label{ga}
 g_{a}(t|x_0,t_0)dt=P(T_a\in dt, T_a<T_{b}|X(t_0)=x_0)
\end{equation}
and
\begin{equation}\label{gb}
 g_{b}(t|x_0,t_0)dt=P(T_b\in dt, T_a>T_{b}|X(t_0)=x_0).
\end{equation} 

For a standard Brownian motion $W=\{W(t),t\geq t_0\}$ the two densities $g_{a}(t|x_0,t_0)$ and $g_{b}(t|x_0,t_0)$ are known in closed form (cf. \cite{Borod}) when the boundaries are constant $a(t)=a$ and $b(t)=b$
\begin{eqnarray}\label{Wiener}
g_{a}(t|x_0,t_0)&=&\sum_{k=-\infty}^{\infty} \frac{x_0-a+2k(b-a)}{\sqrt{2\pi (t-t_0)^3}}e^{-\frac{(x_0-a+2k(b-a))^2}{2(t-t_0)}}\\
g_{b}(t|x_0,t_0)&=&\sum_{k=-\infty}^{\infty}\frac{b-x_0+2k(b-a)}{\sqrt{2\pi (t-t_0)^3}}e^{-\frac{(b-x_0+2k(b-a))^2}{2(t-t_0)}}\nonumber
\end{eqnarray} 
and the pdf and pDf of $T_a$ are
\begin{eqnarray}\label{T_a pdf}
f_{T_a}(t|x_0,t_0)&=&\frac{|a-x_0|}{\sqrt{2\pi (t-t_0)^3}}e^{-\frac{(a-x_0)^2}{2(t-t_0)}}\\
F_{T_a}(t|x_0,t_0)&=&1-\text{erf}\left(\frac{|a-x_0|}{\sqrt{2t}}\right).
\end{eqnarray}

The quantities (\ref{ga}) and (\ref{gb}) are useful for the computation of the joint density function of $T_a$ and $T_b$.
Two different instances arise according to the location of the starting point $X(t_0)=x_0$ with respect to the boundaries (cf. Figure \ref{Fig:2cases}). 
\begin{figure}[htp]
\centering
\includegraphics[height=6cm]{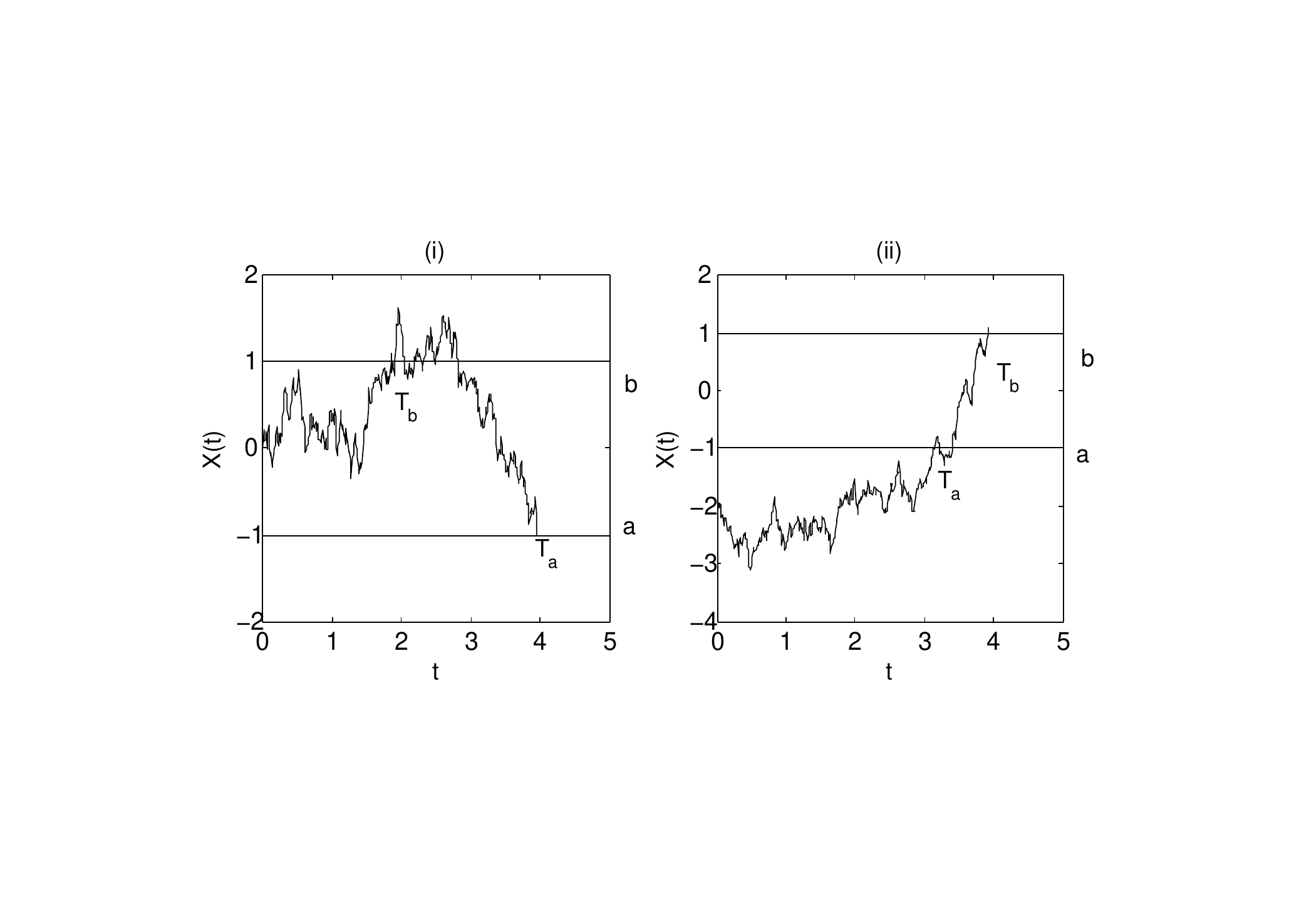}
\caption{Sample path of a stochastic process through two constant boundaries $a$ and $b$: i) $x_0\in (a,b)$; ii) $x_0\notin (a,b)$.}
\label{Fig:2cases}       
\end{figure}
It holds: 

\begin{theorem}\label{lem_general}
Let $X=\{X(t),t\geq t_0\}$ be a diffusion process such that $X(t_0)=x_0$ and let $a(t)$ and $b(t)$ be two continuous time dependent boundaries. 
\begin{enumerate}
	\item[i)] If $x_0<a(t_0)<b(t_0)$ and $a(t)<b(t)$ for each $t>t_0$ or $b(t_0)<a(t_0)<x_0$ and $b(t)<a(t)$ for each $t>t_0$, then
	\begin{equation}\label{densit1}
 f_{T_{a},T_{b}}\left(t,s|x_0,t_0\right)=\left\{
 \begin{array}{ll}
 0 &t\geq s \\
 f_{T_{a}}(t|x_0,t_0)f_{T_{b}}(s|a(t),t)  &t<s  \\
 \end{array} .\right.
 \end{equation}

	\item[ii)] If $a(t_0)<x_0<b(t_0)$ and $a(t)<b(t)$ for each $t>t_0$, then
\begin{equation}\label{densit2}
f_{T_{a} T_{b}}\left(t,s|x_0,t_0\right)=\left\{
\begin{array}{ll}
f_{T_{b}}(s|a(t),t) g_{a}(t|x_0,t_0)  &t<s  \\
0 &t=s \\
f_{T_{a}}(t|b(s),s) g_{b}(s|x_0,t_0)  &t>s \\
\end{array}
.\right.
\end{equation}
 \end{enumerate} 
\end{theorem}
We omit the proof that is straightforward using the strong Markov property.
\begin{remark}
Note that the FPT pdf verifies the initial condition 
\begin{eqnarray}\label{limit}
\lim_{s\rightarrow t} f_{T_{b}}(s|a(t),t)=\lim_{t\rightarrow s}f_{T_{a}}(t|b(s),s)=0.
\end{eqnarray}
Furthermore, due to the differentiability of the boundaries, it holds (cf. \cite{RSS})
\begin{eqnarray}
\lim_{t\rightarrow s}F_{X(t)}(b(t)|b(s),s)&=&\lim_{t\rightarrow s}F_{X(t)}(a(t)|a(s),s)=\frac{1}{2}\label{limit Fa}\\
\lim_{t\rightarrow s}[1-F_{X(t)}(b(t)|a(s),s)]&=&\lim_{t\rightarrow s}F_{X(t)}(a(t)|b(s),s)=0\label{limit Fb}.
\end{eqnarray}  
\end{remark}

\begin{remark}
For some applications it might be interesting to determine the copula function (cf.\cite{Nelsen}) between $T_a$ and $T_b$. When the two densities $g_{a}(t|x_0,t_0)$ and $g_{b}(t|x_0,t_0)$ are known, use of (\ref{densit1}) and (\ref{densit2}) allows to determine this function.
\end{remark}

\section{System of integral equations}\label{Sect:3} 

The computation of $f_{T_{a} T_{b}}\left(t,s|x_0,t_0\right)$ involves the transition pdfs $f_{T_{b}}(s|a(t),t)$ and $f_{T_{a}}(t|b(s),s)$ and the terms $g_{a}(t|x_0,t_0)$ and $g_{b}(t|x_0,t_0)$. 
When the process is a linear regular diffusion, the transition pdf is available in closed form and, if the process is strictly linear, it is Gaussian. In the literature, transition pdf is also available for other regular diffusion processes, such as the Cox-Ingersoll-Ross model (also known as Feller process), the Bessel process or some instances of the Raleigh process (cf. \cite{GS, GNRS}). Further examples arise from space-time transformation of the Brownian motion (cf. \cite{R1}) or of the Cox-Ingersol-Ross process (cf. \cite{CR}). When closed form solutions are not available, the transition pdf is evaluated resorting to numerical methods, such as the numerical solution of the Kolmogorov equation \cite{LP, Sm1} or the numerical inversion of Fourier transforms \cite{VL}.
Unfortunately closed form expressions for the densities $g_{a}(t|x_0,t_0)$ and $g_{b}(t|x_0,t_0)$ are known only for the Brownian motion with constant boundaries (cf. \cite{Borod} formula 3.0.6) or for processes related to it through suitable transformations. Use of the  following theorem helps to overcome this problem.
\begin{theorem}
Let $X=\{X(t),t\geq t_0\}$ be a diffusion process such that $X(t_0)=x_0$. Let $a(t)$ and $b(t)$ be two time dependent boundaries with bounded derivatives such that $a(t_0)<x_0<b(t_0)$ and $a(t)<b(t)$ for each $t>t_0$. 
The pdf's $g_{a}(t|x_0,t_0)$ and $g_{b}(t|x_0,t_0)$ are solution of the following system of Volterra first kind integral equations
\begin{subequations}\label{Fortet2}
\begin{align}
 1-F_{X(t)}(b(t)|x_0,t_0)&=\int_{t_0}^t \left[1-F_{X(t)}(b(t)|a(\tau),\tau)\right]g_{a}(\tau|x_0,t_0) d\tau\label{Fortet2a}\\
 &+\int_{t_0}^t \left[1-F_{X(t)}(b(t)|b(\tau),\tau)\right]g_{b}(\tau|x_0,t_0)d\tau \nonumber\\
F_{X(t)}(a(t)|x_0,t_0)&=\int_{t_0}^t F_{X(t)}(a(t)|a(\tau),\tau)g_{a}(\tau|x_0,t_0) d\tau\label{Fortet2b}\\
 &+ \int_{t_0}^t F_{X(t)}(a(t)|b(\tau),\tau)g_{b}(\tau|x_0,t_0)d\tau.\nonumber
\end{align}
\end{subequations}
\end{theorem}
\begin{proof}
For $t\in [0,\infty]$, conditioning on the boundary first attained by the process, for $x\notin[a(t),b(t)]$  we get
\begin{eqnarray}\nonumber
P\left(X(t)\leq x|X(t_0)=x_0\right)
 &=& \int^{t}_{t_0}P\left(X(t)\leq x|T_{a}<T_{b},T_{a}=\tau,X(t_0)=x_0\right)g_{a}(\tau|x_0,t_0) d\tau\nonumber\\
 &+&\int^{t}_{t_0}P\left(X(t)\leq x|T_{a}>T_{b},T_{b}=\tau,X(t_0)=x_0\right)g_{b}(\tau|x_0,t_0)d\tau\nonumber\\
 &=& \int^{t}_{t_0}P\left(X(t)\leq x|X(\tau)=a(\tau),X(t_0)=x_0\right)g_{a}(\tau|x_0,t_0) d\tau\nonumber\\
 &+&\int^{t}_{t_0}P\left(X(t)\leq x|X(\tau)=b(\tau),X(t_0)=x_0\right)g_{b}(\tau|x_0,t_0)d\tau\nonumber.
\end{eqnarray} 
Differentiating with respect to $x$ we obtain
\begin{eqnarray}\label{Fortetdensity}
f_{X(t)}(x|x_0,t_0)&=&\int_{t_0}^t f_{X(t)}(x|a(\tau),\tau)g_{a}(\tau|x_0,t_0) d\tau\\
 &+& \int_{t_0}^t f_{X(t)}(x|b(\tau),\tau)g_{b}(\tau|x_0,t_0)d\tau.\nonumber
\end{eqnarray} 
Integrating (\ref{Fortetdensity}) with respect to $x$ on the two subdomains $[b(t),\infty]$ and $[-\infty,a(t)]$ respectively, we get (\ref{Fortet2a}) and (\ref{Fortet2b}).
\end{proof}

\begin{remark}
Differentiating (\ref{Fortet2a}) and (\ref{Fortet2b}) with respect to $t$ and recalling \eqref{limit Fa} and \eqref{limit Fb} one gets
\begin{subequations}\label{Fortet2_2}
\begin{align}
 g_{b}(t|x_0,t_0)&=-2\frac{\partial F_{X(t)}(b(t)|x_0,t_0)}{\partial t}\label{Fortet2a_2}\\
 &+\int_{t_0}^t 2\left(\frac{\partial F_{X(t)}(b(t)|a(\tau),\tau)}{\partial t}g_{a}(\tau|x_0,t_0)+ \frac{\partial F_{X(t)}(b(t)|b(\tau),\tau)}{\partial t}g_{b}(\tau|x_0,t_0)\right)d\tau.\nonumber\\
  g_{a}(t|x_0,t_0)&=2\frac{\partial F_{X(t)}(a(t)|x_0,t_0)}{\partial t}\label{Fortet2b_2}\\
 &-\int_{t_0}^t 2\left(\frac{\partial F_{X(t)}(a(t)|a(\tau),\tau)}{\partial t}g_{a}(\tau|x_0,t_0)+ \frac{\partial F_{X(t)}(a(t)|b(\tau),\tau)}{\partial t}g_{b}(\tau|x_0,t_0)\right)d\tau.\nonumber
 \end{align}
\end{subequations}
This system of Volterra integral equations coincides with the one proposed in \cite{BGNR} if the kernel of the two equations is regularized.
\end{remark}

It holds
\begin{theorem}
The system (\ref{Fortet2}) has a unique continuous solution for $t>t_0$.
\end{theorem}
\begin{proof}
The system of Volterra integral equations of the first kind (\ref{Fortet2}) is equivalent to the system of Volterra integral equations of the second kind (\ref{Fortet2_2}) that can be written in matricial form
\begin{equation}
\text{\bfseries{g}}(t)=\text{\bfseries{h}}(t)+\int_{t_0}^t \text{\bfseries{k}}(t,\tau)\text{\bfseries{g}}(\tau) d\tau
\end{equation}
where 
\begin{equation}
\text{\bfseries{g}}(t)=\left[
\begin{array}{c}
g_{a}(t|x_0,t_0) \\
g_{b}(t|x_0,t_0)\end{array}
\right],
\end{equation}
\begin{equation}
\text{\bfseries{h}}(t)=\left[
\begin{array}{c}
-2\frac{\partial F_{X(t)}(b(t)|x_0,t_0)}{\partial t} \\
2\frac{\partial F_{X(t)}(a(t)|x_0,t_0)}{\partial t}\end{array}
\right],
\end{equation}
\begin{equation}
\text{\bfseries{k}}(t,\tau)=\left[
\begin{array}{cc}
-2\frac{\partial F_{X(t)}(b(t)|a(\tau),\tau)}{\partial t} & -2\frac{\partial F_{X(t)}(b(t)|b(\tau),\tau)}{\partial t}\\
2\frac{\partial F_{X(t)}(a(t)|a(\tau),\tau)}{\partial t}& 2\frac{\partial F_{X(t)}(a(t)|b(\tau),\tau)}{\partial t}
\end{array}
\right],
\end{equation}
Since the kernel $\text{\bfseries{k}}(t,\tau)$ is singular in $\tau=t$, we introduce an equivalent system with continuous kernel. Mimicking the method presented in \cite{BNR}, we introduce two couple of functions $\gamma_i(t)$ and $\eta_i(t)$, $i=1,2$, continuous in $[t_0,+\infty]$. Combining (\ref{Fortet2}), (\ref{Fortetdensity}) and (\ref{Fortet2_2}), together with $\gamma_i(t)$ and $\eta_i(t)$ we obtain
\begin{eqnarray}\label{eqconKR}
 g_{b}(t|x_0,t_0)=-\Psi^1(b(t)|x_0,t_0)+\int_{t_0}^t \left( \Psi^1(b(t)|a(\tau),\tau)g_{a}(\tau|x_0,t_0)+ \Psi^1(b(t)|b(\tau),\tau)g_{b}(\tau|x_0,t_0)\right)d\tau.\nonumber\\
  g_{a}(t|x_0,t_0)=\Psi^2(a(t)|x_0,t_0)-\int_{t_0}^t \left( \Psi^2(a(t)|a(\tau),\tau)g_{a}(\tau|x_0,t_0)+ \Psi^2(a(t)|b(\tau),\tau)g_{b}(\tau|x_0,t_0)\right)d\tau.
\end{eqnarray}
where
\begin{eqnarray}\label{Psi}
 \Psi^1(b(t)|x,s)&=&-2\frac{\partial F_{X(t)}(b(t)|x,s)}{\partial t}+\gamma_1(t)f_{X(t)}(b(t),t|x,s)+\eta_1(t) [1-F_{X(t)}(b(t)|x,s)]\nonumber\\
  \Psi^2(a(t)|x,s)&=&2\frac{\partial F_{X(t)}(a(t)|x,s)}{\partial t}-\gamma_2(t)f_{X(t)}(a(t),t|x,s)-\eta_2(t) F_{X(t)}(a(t)|x,s)
\end{eqnarray}

A suitable choice of $\gamma_i(t)$ and $\eta_i(t)$, $i=1,2$, makes $\Psi^1(b(t)|b(\tau),\tau)$ and $\Psi^2(a(t)|a(\tau),\tau)$ not singular. On the other hand, since (cf. \cite{BGNR})
\begin{eqnarray}
\lim_{\tau\rightarrow t}f_{X(t)}(b(t),t|a(\tau),\tau)&=\left.\lim_{\tau\rightarrow t}\frac{\partial}{\partial x}f_{X(t)}(x,t|a(\tau),\tau)\right|_{x=b(t)}=0\\
\lim_{\tau\rightarrow t}f_{X(t)}(a(t),t|b(\tau),\tau)&=\left.\lim_{\tau\rightarrow t}\frac{\partial}{\partial x}f_{X(t)}(x,t|b(\tau),\tau)\right|_{x=a(t)}=0\nonumber,
\end{eqnarray}
the kernels $\Psi^1(b(t)|a(\tau),\tau)$ and $\Psi^2(a(t)|b(\tau),\tau)$ are not singular. This makes possible to apply Theorem 3.11 of \cite{Li} to get the thesis.
\end{proof}

\begin{remark}
The functions $\gamma_i(t)$ and $\eta_i(t)$, $i=1,2$ can be determined. For example, for an Ornstein Uhlenbeck process characterized by the drift $\mu(t,x)=\alpha x+\beta$ and infinitesimal variance $\sigma(t,x)=\sigma$, where $\alpha$, $\beta$ and $\sigma>0$ are arbitrary real constants. The functions that regularize the kernels are
$\gamma_1(t)=1/2[\alpha b(t)+\beta -b'(t)]$, $\gamma_2(t)=1/2[\alpha a(t)+\beta -a'(t)]$ and $\eta_i(t)\equiv0$, $i=1,2$ (cf. \cite{BNR}).
\end{remark}
When the boundaries are constant it holds:
\begin{corollary}\label{Cor Laplace}
Let $X=\{X(t),t\geq t_0\}$ be a diffusion process such that $X(t_0)=x_0$ and let $a$ and $b$ be two constant boundaries such that $a<x_0<b$, then the following three expressions are equivalent for $g_{a}^{\lambda}(x_0)=\int_{t_0}^{+\infty}e^{-\lambda t}g_{a}(t|x_0,t_0)dt$ and $g_{b}^{\lambda}(x_0)=\int_{t_0}^{+\infty}e^{-\lambda t}g_{b}(t|x_0,t_0)dt$:
\begin{subequations}\label{LaplaceItoMcKean}
\begin{align}
g_{a}^{\lambda}(x_0)&=\frac{f_{T_b}^{\lambda}(x_0)f_{T_a}^{\lambda}(b)-f_{T_a}^{\lambda}(x_0)}{f_{T_a}^{\lambda}(b)f_{T_b}^{\lambda}(a)-1} \label{LaplaceOtoMcKeana}\\
g_{b}^{\lambda}(x_0)&=\frac{f_{T_a}^{\lambda}(x_0)f_{T_b}^{\lambda}(a)-f_{T_b}^{\lambda}(x_0)}{f_{T_a}^{\lambda}(b)f_{T_b}^{\lambda}(a)-1} \label{LaplaceOtoMcKeanb}
\end{align}
\end{subequations}
\begin{subequations}\label{LaplaceFortet2}
\begin{align}
g_{a}^{\lambda}(x_0)&=\frac{\left[1-\lambda F_X^{\lambda}(b|x_0)\right]F_X^{\lambda}(a|b)-\left[1-\lambda F_X^{\lambda}(b|b)\right]F_X^{\lambda}(a|x_0)}{\left[1-\lambda F_X^{\lambda}(b|a)\right]F_X^{\lambda}(a|b)-\left[1-\lambda F_X^{\lambda}(b|b)\right]F_X^{\lambda}(a|a)} \label{LaplaceFortet2a}\\
g_{b}^{\lambda}(x_0)&=\frac{\left[1-\lambda F_X^{\lambda}(b|x_0)\right]F_X^{\lambda}(a|a)-\left[1-\lambda F_X^{\lambda}(b|a)\right]F_X^{\lambda}(a|x_0)}{\left[1-\lambda F_X^{\lambda}(b|b)\right]F_X^{\lambda}(a|a)-\left[1-\lambda F_X^{\lambda}(b|a)\right]F_X^{\lambda}(a|b)} \label{LaplaceFortet2b}
\end{align}
\end{subequations}
\begin{subequations}\label{Laplacedens}
\begin{align}
g_{a}^{\lambda}(x_0)&=\frac{f_X^{\lambda}(x_1|x_0)f_X^{\lambda}(x_2|b)-f_X^{\lambda}(x_1|b)f_X^{\lambda}(x_2|x_0)}{f_X^{\lambda}(x_1|a)f_X^{\lambda}(x_2|b)-f_X^{\lambda}(x_1|b)f_X^{\lambda}(x_2|a)}\nonumber\\
&=\frac{v_1(\alpha,x_0)v_2(\alpha,b)-v_2(\alpha,x_0)v_1(\alpha,b)}{v_1(\alpha,a)v_2(\alpha,b)-v_2(\alpha,a)v_1(\alpha,b)} \label{Laplacedensa}\\
g_{b}^{\lambda}(x_0)&=\frac{f_X^{\lambda}(x_1|a)f_X^{\lambda}(x_2|x_0)-f_X^{\lambda}(x_1|x_0)f_X^{\lambda}(x_2|a)}{f_X^{\lambda}(x_1|a)f_X^{\lambda}(x_2|b)-f_X^{\lambda}(x_1|b)f_X^{\lambda}(x_2|a)}\nonumber\\
&=\frac{v_1(\alpha,a)v_2(\alpha,x_0)-v_2(\alpha,a)v_1(\alpha,x_0)}{v_1(\alpha,a)v_2(\alpha,b)-v_2(\alpha,a)v_1(\alpha,b)} \label{Laplacedensb}
\end{align}
\end{subequations}
 where 
 \begin{eqnarray*}
 F_X^{\lambda}(x|x_0)&=&\int_0^{+\infty}e^{-\lambda t}F_{X(t)}(x|x_0,t_0)dt\\
 f_X^{\lambda}(x|x_0)&=&\int_0^{+\infty}e^{-\lambda t}f_{X(t)}(x|x_0,t_0)dt\\
 f_{T_a}^{\lambda}(x_0)&=&\int_0^{+\infty}e^{-\lambda t}F_{T_a}(t|x_0,t_0)dt
 \end{eqnarray*}
and the functions $v_i(\alpha,x)$, $i=1,2$ are fundamental solutions of (8.13b) in \cite{RicSat90}.
\end{corollary}
\begin{proof}
Generalizing the standard calculation of p. 30 in \cite{IM} for an arbitrary regular diffusion we obtain (\ref{LaplaceItoMcKean}).

\noindent Applying Laplace transform to (\ref{Fortet2}) and using the convolution theorem we get (\ref{LaplaceFortet2}), a result recently published in \cite{GNR}.

\noindent Applying Laplace transform to (\ref{Fortetdensity}) together with the convolution theorem for two generic points $x_1>b$ and $x_2<a$ we get the first equality in (\ref{Laplacedens}). The use of (8.22) in \cite{RicSat90} allows to get the second equality.

\noindent Furthermore, recalling that (cf. \cite{R})
	\[f_{T_a}^{\lambda}(x_1|x_0)=\frac{f_X^{\lambda}(x|x_0)}{f_X^{\lambda}(x|a)},
\]
the first equality in (\ref{Laplacedens}) becomes (\ref{LaplaceItoMcKean}).

\noindent Finally, remembering that (cf. \cite{R})
\begin{eqnarray}
f_{T_a}^{\lambda}(x_1|x_0)=\left\{
\begin{array}{lll}
\frac{1-\lambda F_X^{\lambda}(a|x_0)}{1-\lambda F_X^{\lambda}(a|a)} &\text{if} 	 &a>x_0\\
\frac{F_X^{\lambda}(a|x_0)}{F_X^{\lambda}(a|a)} &\text{if} &a<x_0
\end{array}\right.
\end{eqnarray}
the equations (\ref{LaplaceItoMcKean}) become (\ref{LaplaceFortet2}).
This implies that the three formulations are equivalent.
\end{proof}

\begin{remark}
The above results also hold for diffusion processes bounded by one or two reflecting boundaries when the diffusion interval is characterized by non natural boundaries, i.e. for Cox-Ingersoll-Ross whose diffusion interval is $I=[0,\infty)$ or for the reflected Brownian Motion.
\end{remark}

\section{Algorithms for $P(T_a\in dt,T_a<T_b)$ and $P(T_b\in dt,T_a>T_b)$}\label{Sect:4} 

In this section we describe two approaches to determine the density functions $g_{a}(t|x_0,t_0)$ and $g_{b}(t|x_0,t_0)$.

When the boundaries are constant the densities $g_{a}(t|x_0,t_0)$ and $g_{b}(t|x_0,t_0)$ are obtained from the Laplace transforms (\ref{LaplaceFortet2}) by inverting them numerically using, for example, Euler method \cite{Abate}.

Alternative methods become necessary when the boundaries $a(t)$ and $b(t)$ are time dependent or when the Laplace inversion presents numerical difficulties.  For example in the case of the Ornstein Uhlenbeck process the expression of $g_{a}^{\lambda}(x_0)$ and $g_{b}^{\lambda}(x_0)$ involve parabolic cylinder function (cf. \cite{GS1}). Their numerical inversion requests efforts specific for this instance. Furthermore there are processes for which $F_X^{\lambda}(a|x_0)$ and $F_X^{\lambda}(b|x_0)$ are not known in the literature. Their computation is possible however it requests the solution of specific second order differential equations (cf. \cite{R}).

Here we propose the following numerical method that can be applied both for constant and time depending boundaries.
Let us introduce a time discretization $t_i=t_0+ih$, $i=1,2,\ldots$ where $h$ is a positive constant. To determine the two pdf's $g_{a}(t|x_0,t_0)$ and $g_{b}(t|x_0,t_0)$ at the finite set of knots $t_i$ for $i=1,\ldots,n$, we use Euler method \cite{Li} to approximate the integrals on the r.h.s. of (\ref{Fortet2a}) and (\ref{Fortet2b}). Hence we get
	\begin{subequations}\label{Fortet2discret}
 \begin{align}
 1-F_{X(t_i)}(b(t_i)|x_0,t_0)&=\sum_{j=1}^i \left[1-F_{X(t_i)}(b(t_i)|a(t_j),t_j)\right]\hat{g}_{a}(t_j|x_0,t_0) h\label{Fortet2discreta}\\
 &+\sum_{j=1}^i \left[1-F_{X(t_i)}(b(t_i)|b(t_j),t_j)\right]\hat{g}_{b}(t_j|x_0,t_0)h \nonumber\\
 F_{X(t_i)}(a(t_i)|x_0,t_0)&=\sum_{j=1}^i F_{X(t_i)}(a(t_i)|a(t_j),t_j)\hat{g}_{a}(t_j|x_0,t_0) h\label{Fortet2discretb}\\
 &+ \sum_{j=1}^i F_{X(t_i)}(a(t_i)|b(t_j),t_j)\hat{g}_{b}(t_j|x_0,t_0)h.\nonumber
 \end{align}
\end{subequations}
The densities $g_{a}(t|x_0,t_0)$ and $g_{b}(t|x_0,t_0)$ can be evaluated in the knots $t_i$ for $i=1,\ldots,n$ by means of the following algorithm.

{\bf Step 1}
\begin{subequations}\label{step1}
\begin{align}
\hat{g}_{b}(t_1|x_0,t_0)&=\frac{2(1-F_{X(t_1)}(b(t_1)|x_0,t_0))}{h}\\
\hat{g}_{a}(t_1|x_0,t_0)&=\frac{2F_{X(t_1)}(a(t_1)|x_0,t_0)}{h}.
\end{align}
\end{subequations}

{\bf Step $i$, $i=2,3,\ldots$}
\begin{subequations}\label{step i 1}
\begin{align}
\hat{g}_{b}(t_i|x_0,t_0)&=\frac{2(1-F_{X(t_i)}(b(t_i)|x_0,t_0))}{h}\\
&-2\sum _{j=1}^{i-1}\left[(1-F_{X(t_i)}(b(t_i)|a(t_j),t_j))\hat{g}_{a}(t_j|x_0,t_0)\right.\nonumber\\
&+\left. (1-F_{X(t_i)}(b(t_i)|b(t_j),t_j))\hat{g}_{b}(t_j|x_0,t_0)  \right]\nonumber\\
\hat{g}_{a}(t_i|x_0,t_0)&=\frac{2F_{X(t_i)}(a(t_i)|x_0,t_0)}{h}\\
&-2\sum _{j=1}^{i-1}\left[F_{X(t_i)}(a(t_i)|a(t_j),t_j)\hat{g}_{a}(t_j|x_0,t_0)\right.\nonumber\\
&+ \left.F_{X(t_i)}(a(t_i)|b(t_j),t_j)\hat{g}_{b}(t_j|x_0,t_0) \right]\nonumber
\end{align}
\end{subequations}

where we used \eqref{limit Fa} and \eqref{limit Fb}.

\begin{remark}
The choice of equally spaced knots is motivated by the simplification of the notation but the method can be easily extended to non constant $h$.
\end{remark}
 
\begin{theorem}\label{error}
If constants $c_1$ and $c_2$ exist, such that for all $h>0$
\begin{eqnarray}
\max_{1<i<n, 0<j\leq i-1}, |F_{X(t_i)}(b(t_i)|a(t_j),t_j)-F_{X(t_{i-1})}(b(t_{i-1})|a(t_j),t_j)|\leq c_1 h\label{ip1}\\ 
\max_{1<i<n, 0<j\leq i-1}, |F_{X(t_i)}(a(t_i)|b(t_j),t_j)-F_{X(t_{i-1})}(a(t_{i-1})|b(t_j),t_j)|\leq c_2 h\label{ip2}
\end{eqnarray}
then the absolute value of the errors $\epsilon_{a,i}$ and $\epsilon_{b,i}$ of the proposed algorithm at the discretization knots $t_i$, $i = 1, 2,\ldots$
\begin{eqnarray*}
\epsilon_{a,i}&:=&\hat{g}_{a}(t_i|x_0,t_0)-g_{a}(t_i|x_0,t_0)\\
\epsilon_{b,i}&:=&\hat{g}_{b}(t_i|x_0,t_0)-g_{b}(t_i|x_0,t_0)
\end{eqnarray*}
are $O(h)$.
\end{theorem}

\begin{proof}
The Euler method applied to (\ref{Fortet2}) gives
\begin{subequations}\label{error1}
\begin{align}
  1-F_{X(t_i)}(b(t_i)|x_0,t_0)&=\sum_{j=1}^i \left[1-F_{X(t_i)}(b(t_i)|a(t_j),t_j)\right]g_{a}(t_j|x_0,t_0) h\label{error1a}\\
  &+\sum_{j=1}^i \left[1-F_{X(t_i)}(b(t_i)|b(t_j),t_j)\right]g_{b}(t_j|x_0,t_0)h \nonumber\\
 &+\delta_1(h,t_n)\nonumber\\
 F_{X(t_i)}(a(t_i)|x_0,t_0)&=\sum_{j=1}^i F_{X(t_i)}(a(t_i)|a(t_j),t_j)g_{a}(t_j|x_0,t_0) h\label{error1b}\\
 &+ \sum_{j=1}^i F_{X(t_i)}(a(t_i)|b(t_j),t_j)g_{b}(t_j|x_0,t_0)h\nonumber\\
 &+\delta_2(h,t_n).\nonumber
\end{align}
\end{subequations}
where $\delta_1(h,t_i)$ and $\delta_2(h,t_i)$ are the differences between the integrals on the r.h.s. of (\ref{Fortet2}) and the finite sums on the r.h.s. of (\ref{error1}).

On subtracting \eqref{error1a} from \eqref{Fortet2discreta} and \eqref{error1b} from \eqref{Fortet2discretb}we get
\begin{subequations}\label{error2}
\begin{align}
  \delta_1(h,t_i)&=h\sum_{j=1}^i \left[1-F_{X(t_i)}(b(t_i)|a(t_j),t_j)\right]\epsilon_{a,j} +h\sum_{j=1}^i \left[1-F_{X(t_i)}(b(t_i)|b(t_j),t_j)\right]\epsilon_{b,j}\\
\delta_2(h,t_i)&=h\sum_{j=1}^i F_{X(t_i)}(a(t_i)|a(t_j),t_j)\epsilon_{a,j}+ h\sum_{j=1}^i F_{X(t_i)}(a(t_i)|b(t_j),t_j)\epsilon_{b,j}.
\end{align}
\end{subequations}
Differencing \eqref{error2} and recalling \eqref{limit Fa} and \eqref{limit Fb} we obtain
\begin{subequations}\label{error3}
\begin{align}
  \delta_1(h,t_i)-\delta_1(h,t_{i-1})&=\frac{h}{2}\epsilon_{b,i}+ h \sum_{j=1}^{i-1} \left(\left[1-F_{X(t_i)}(b(t_i)|a(t_j),t_j)\right]-\left[1-F_{X(t_{i-1})}(b(t_{i-1})|a(t_j),t_j)\right]\right)\epsilon_{a,j}\nonumber\\ &+h\sum_{j=1}^{i-1} \left(\left[1-F_{X(t_i)}(b(t_i)|b(t_j),t_j)\right]-\left[1-F_{X(t_{i-1})}(b(t_{i-1})|b(t_j),t_j)\right]\right)\epsilon_{b,j}\\
\delta_2(h,t_i)-\delta_2(h,t_{i-1})&=\frac{h}{2}\epsilon_{a,i}+ h\sum_{j=1}^{i-1} \left(F_{X(t_i)}(a(t_i)|a(t_j),t_j)-F_{X(t_{i-1})}(a(t_{i-1})|a(t_j),t_j)\right)\epsilon_{a,j}\nonumber\\
&+ \sum_{j=1}^{i-1} \left(F_{X(t_i)}(a(t_i)|b(t_j),t_j)-F_{X(t_{i-1})}(a(t_{i-1})|b(t_j),t_j)\right)\epsilon_{b,j}
\end{align}
\end{subequations}
that can be rewritten as
\begin{subequations}\label{error4}
\begin{align}
\epsilon_{b,i}&=-2\sum_{j=1}^{i-1} \left(\left[1-F_{X(t_i)}(b(t_i)|a(t_j),t_j)\right]-\left[1-F_{X(t_{i-1})}(b(t_{i-1})|a(t_j),t_j)\right]\right)\epsilon_{a,j}\\
&-2\sum_{j=1}^{i-1} \left(\left[1-F_{X(t_i)}(b(t_i)|b(t_j),t_j)\right]-\left[1-F_{X(t_{i-1})}(b(t_{i-1})|b(t_j),t_j)\right]\right)\epsilon_{b,j}\nonumber\\
&+\frac{2}{h}\left( \delta_1(h,t_i)-\delta_1(h,t_{i-1})\right)\nonumber\\
\epsilon_{a,i}&=-2\sum_{j=1}^{i-1} \left(F_{X(t_i)}(a(t_i)|a(t_j),t_j)-F_{X(t_{i-1})}(a(t_{i-1})|a(t_j),t_j)\right)\epsilon_{a,j}\\
&-2\sum_{j=1}^{i-1} \left(F_{X(t_i)}(a(t_i)|b(t_j),t_j)-F_{X(t_{i-1})}(a(t_{i-1})|b(t_j),t_j)\right)\epsilon_{b,j}\nonumber\\
&+\frac{2}{h}\left(\delta_2(h,t_i)-\delta_2(h,t_{i-1})\right)\nonumber.
\end{align}
\end{subequations}
Let us now consider the global error
\begin{equation}
\xi_i=|\epsilon_{a,i}|+|\epsilon_{b,i}|.
\end{equation}
When the hypotheses \eqref{ip1} and \eqref{ip2} are fulfilled
\begin{eqnarray}
|\xi_i|&\leq&\left|(c_1+c_2)h\right|\sum_{j=1}^{i-1}\left|\xi_j\right|\\
&+&\frac{2}{h}\left| |\delta_1(h,t_i)-\delta_1(h,t_{i-1})|+|\delta_2(h,t_i)-\delta_2(h,t_{i-1})|\right|\nonumber.
\end{eqnarray}
Observing that Euler method errors are $|\delta_1(h,t)|=|\delta_2(h,t)|=O(h^2)$ and applying Theorem 7.1 of \cite{Li} we get $|\xi_i|=O(h^2)$ and hence the thesis.
\end{proof}

\begin{remark}
A better result on the errors can be obtained by improving the integral discretization rule, i.e. using the midpoint formula instead of Euler method. Other integration rules can improve the order of the error but strongly increase the computational complexity of the algorithm.
\end{remark}

\begin{remark}
The two methods are equivalent in terms of computational time when the Laplace transform expression is a well behaved function. Nevertheless, the generalization of the method for a time dependent boundary $S(t)$ is possible only for the numerical method.
\end{remark}

\section{Examples}\label{Sect:5} 

In this section we discuss a set of examples of interest for the applications, i.e. standard Brownian motion, Geometric Brownian motion, Ornstein Uhlenbeck process. We apply the algorithms of Section \ref{Sect:4} for numerical evaluations. When the joint densities are known in closed form, we use them to illustrate the reliability of the algorithms. 

\subsection{Standard Brownian motion}\label{SubSect:1} 

Let us consider a standard Brownian motion with constant boundaries. It is a time and space homogeneous diffusion process, hence we can rewrite its joint density functions (\ref{densit1}) and (\ref{densit2}) in closed form as

\begin{enumerate}
\item[i)] If $x_0<a<b$ or $b<a<x_0$ then
	\begin{equation}\label{densit1BM}
 f_{T_{a} T_{b}}\left(t,s\right)=\left\{
 \begin{array}{ll}
 0 &t\geq s \\
 f_{T_{a}}(t|0,0)f_{T_{b-a}}(s-t|0,0)  &t<s  \\
 \end{array} .\right.
 \end{equation}

	\item[ii)] If $a<x_0<b$ then
\begin{equation}\label{densit2BM}
f_{T_{a} T_{b}}\left(t,s\right)=\left\{
\begin{array}{ll}
f_{T_{b-a}}(s-t|0,0) g_{a}(t|0,0)  &t<s  \\
0 &t=s \\
f_{T_{a-b}}(t-s|0,0) g_{b}(s|0,0)  &t>s \\
\end{array}
.\right.
\end{equation}
\end{enumerate} 
where $g_{a}(t|x_0,t_0)$ and $g_{b}(t|x_0,t_0)$ are given by (\ref{Wiener}) and $f_{T_{a}}(t|0,0)$ is given by (\ref{T_a pdf}).

Figure \ref{Fig:Wiener} shows the joint pdf of the first hitting times of a standard Brownian motion through two constant boundaries $b=-a=1$ and the corresponding copula together with the contour lines.
Figures \ref{Fig:Wiener2} illustrates the case of constant boundaries $a=-1$ and $b=1.5$ asymmetric with respect to $x_0$. The asymmetry of the boundaries location determines peaks of different height. Note that the maximum of the joint density have inverted height in the corresponding copula.

Furthermore the Laplace transforms of $g_{a}^{\lambda}(x_0)$ and $g_{b}^{\lambda}(x_0)$ can be found in \cite{Borod} or applying  Corollary \ref{Cor Laplace}.

\begin{remark}
The stability of the algorithms introduced in Section \ref{Sect:4}, already proved by Theorem \ref{error}, is confirmed by the standard Brownian motion case where the pdf's $g_{a}(t|x_0,t_0)$ and $g_{b}(t|x_0,t_0)$ are available in closed form. We apply the algorithms to the standard Brownian motion with constant boundaries $a=-1$ and $b=2$ with discretization step $h=0.01$ and we compare the results with the closed form densities (\ref{Wiener}) with the series truncated to $N=10^3$ steps. The inversion of the Laplace transform with the Euler method gives a mean square deviation $MSE_a= 6.02\cdot10^{-19}$ and $MSE_b=9.83\cdot10^{-20}$. The numerical algorithm gives a mean square deviation $MSE_a=3.23\cdot10^{-6}$ and $MSE_b=5.11\cdot10^{-8}$. It confirms the reliability of the new algorithm. The higher precision of the Laplace inversion with respect to the numerical method is determined by the simple expression of the involved Laplace transforms. However we cannot infer an analogous property for the other diffusions.
\end{remark}

\begin{remark} \label{Wmu}
The extension of the above results to a Brownian motion with diffusion coefficient $\sigma \neq 1$ is straightforward. Indeed, a Brownian motion with diffusion coefficient $\sigma$ can be transformed in a standard Brownian motion via the space transformation $x=x'/ \sigma$ and the boundaries $a$ and $b$ becomes $a/ \sigma$ and $b/ \sigma$ respectively. 
When $\mu>0$ and $a(0)<x_0<b(0)$ one can determine $g_b(t|x_0,t_0)$. In this case the crossing of the boundary $a(t)$ is not a sure event and the study of $g_a(t|x_0,t_0)$ requests a suitable normalization.
Similarly the case of $\mu<0$ and $a(0)<x_0<b(0)$ is analogous interchanging the role of the two boundaries.
\end{remark}

\begin{remark} 
Indicating with $C^{\sigma}_{T_a,T_b}$ the copula of $(T_a,T_b)$ for a Brownian motion with diffusion coefficient $\sigma$ and with $C_{T_a,T_b}$ the copula in the case $\sigma=1$, recalling the transformation $x=x'/ \sigma$, the relationship $C^{\sigma}_{T_a,T_b}=C_{T_{a/\sigma},T_{b/\sigma}}$ holds.
Geometric Brownian motion can be obtained by a standard Brownian motion via the space transformation $x'=\exp(\sigma x)$. The corresponding copula, $C^{GBM}_{T_a,T_b}$, is related with the copula of the standard Brownian motion through $C^{GBM}_{T_a,T_b}=C_{T_{\ln a /\sigma},T_{\ln a /\sigma }}$.

The more general transformation $x'=\exp(\mu t+\sigma x)$ is not interesting from the point of view of the exit times from a strip because it corresponds to transform the process into a Brownian motion with drift that has not a sure crossing, as stated in Remark \ref{Wmu}.
\end{remark}

\begin{figure}[htp]
\centering
\includegraphics[height=10cm]{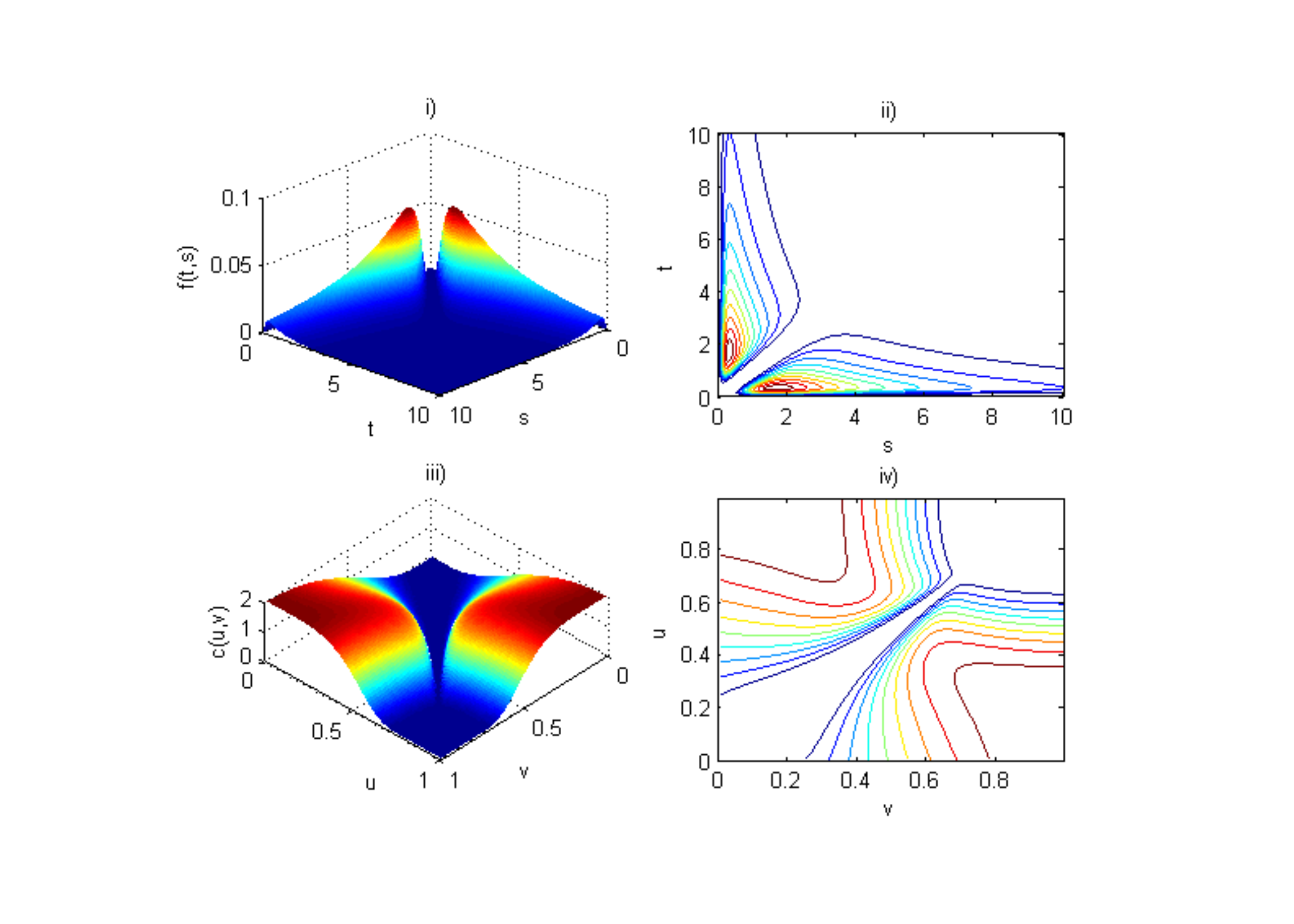}
\caption{First hitting times of a Brownian motion through two constant boundaries $b=-a=1$: i) Joint pdf ii) Contour lines of the joint pdf iii) Density of the copula iv) Contour lines of the density of the copula.}
\label{Fig:Wiener}       
\end{figure}

\begin{figure}[htp]
\centering
\includegraphics[height=10cm]{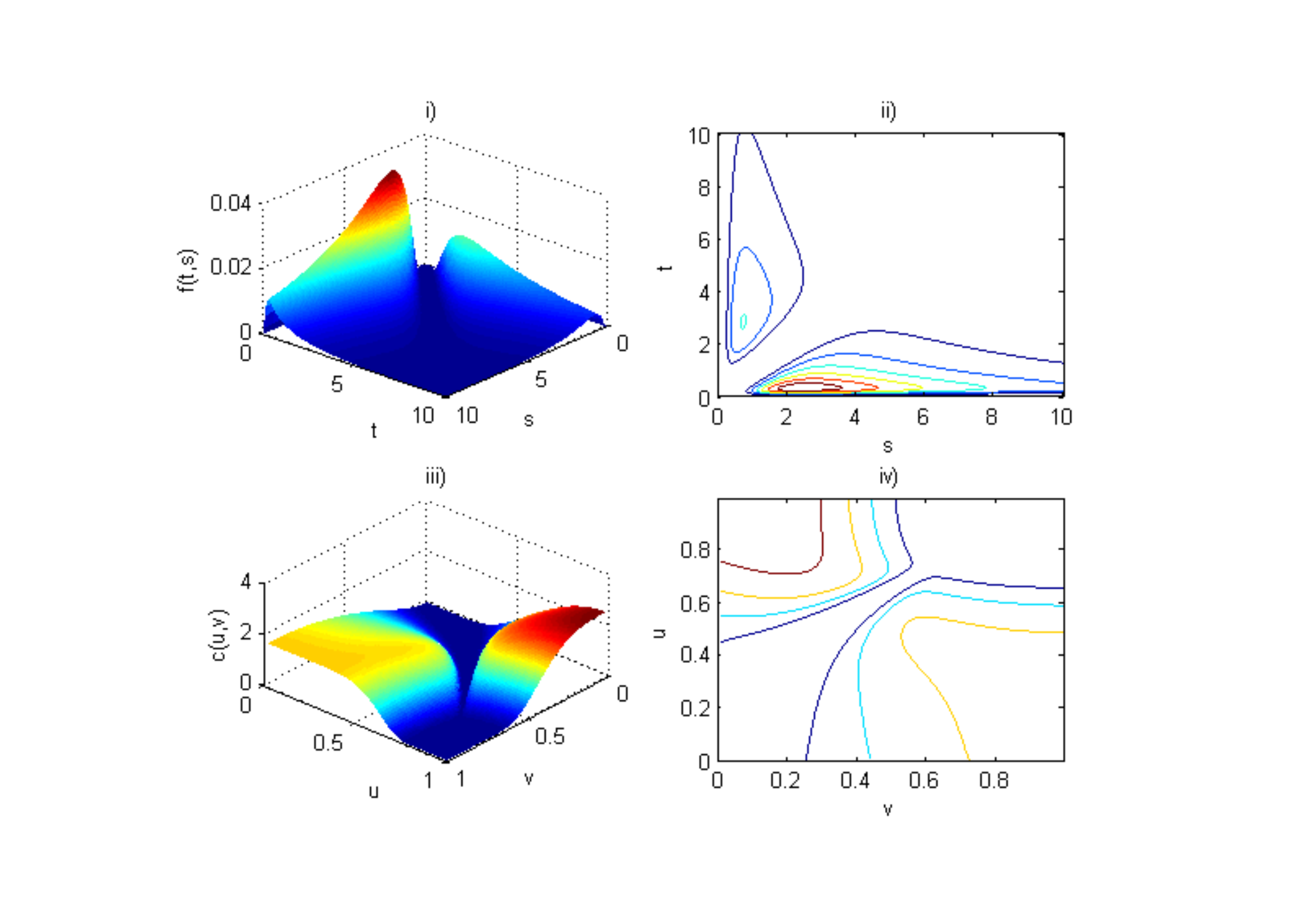}
\caption{First hitting times of a Brownian motion through two constant boundaries $a=-1$ and $b=1.5$: i) Joint pdf ii) Contour lines of the joint pdf iii) Density of the copula iv) Contour lines of the density of the copula.}
\label{Fig:Wiener2}       
\end{figure}

As a further example we consider a standard Brownian motion with the following boundaries $b(t)=1+0.1 \cos (\pi t)$ and $a(t)=-1 +0.1 \cos (\pi t+\pi)$. Since the boundaries are time dependent, Laplace transform inversions cannot be applied. Figure \ref{Fig:Wiener_osc} shows the joint pdf of the first hitting times and the corresponding contour lines obtained with the proposed numerical algorithm.

\begin{figure}[htp]
\centering
\includegraphics[height=8cm]{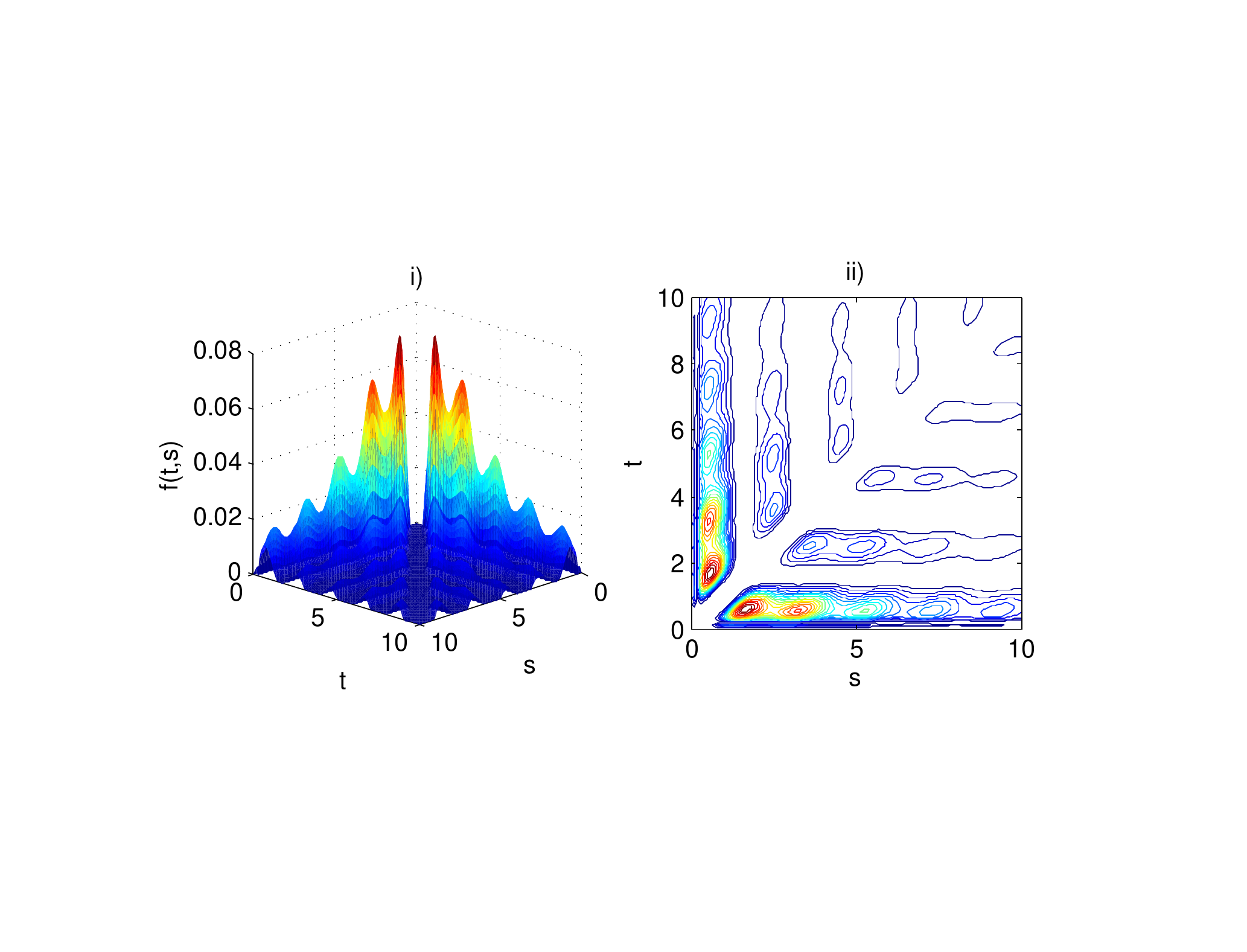}
\caption{First hitting times of a Brownian motion through two time dependent boundaries $b(t)=1+0.1 \cos (\pi t)$ and $a(t)=-1 +0.1 \cos (\pi t+\pi)$: i) Joint pdf ii) Contour lines of the joint pdf .}
\label{Fig:Wiener_osc}       
\end{figure}

\subsection{Ornstein Uhlenbeck Process}\label{SubSect:3} 

Consider as a further example the Ornstein Uhlenbeck process, described by the stochastic differential equation

\begin{eqnarray}\label{OU}
dX(t)&=&\left(-\frac{X(t)}{\theta}+\mu\right)dt+\sigma dW_t\\
X(t_0)&=&x_0.
\end{eqnarray}

For this process representations and numerical methods are available and can be used to evaluate the first hitting time pdf $f_{T_a}(t|x_0,t_0)$ \cite{APP,BNR, L}. On the other side, the density $g_a(t|x_0,t_0)$ is not known in closed form. Here we have applied classical numerical algorithms (cf. \cite{BNR}) to evaluate the first hitting time pdf and the algorithms of Section 4 to compute the second.
Figure \ref{Fig:OU} shows the joint pdf and the corresponding copula of the first hitting times of an Ornstein Uhlenbeck process with parameter $\theta=10$, $\mu=0$, $\sigma=1$ and $x_0=0$ through two constant boundaries $b=-a=1$.
Figures \ref{Fig:OU2} illustrates the case of asymmetric w.r.t. $x_0$ constant boundaries $a=-1$ and $b=1.5$. Note that $x_0$ represents the symmetry axis of the Ornstein Uhlenbeck process. The height of the peaks of the joint density and of the copula behaves as the Brownian motion case. 

\begin{figure}[htp]
\centering
\includegraphics[height=10cm]{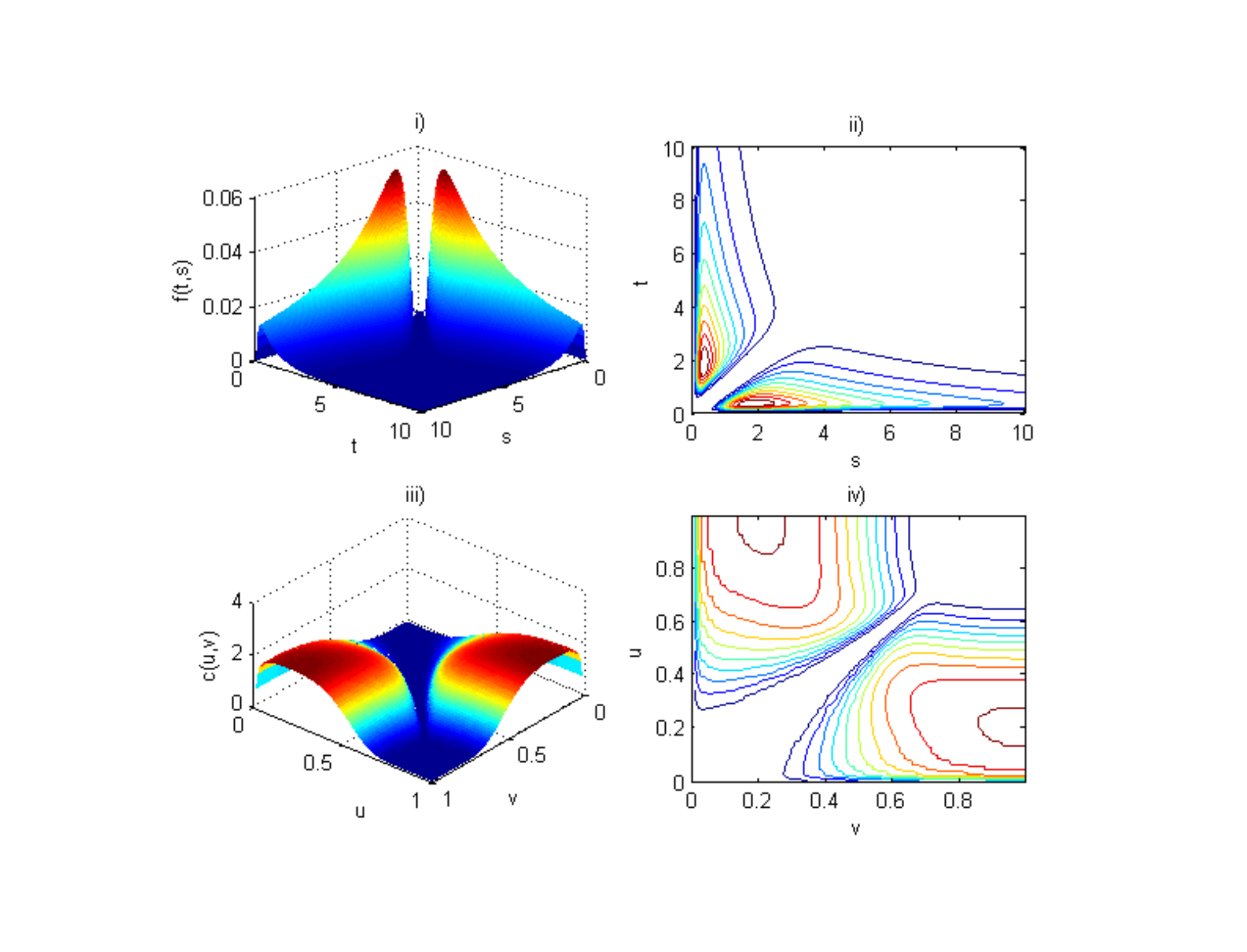}
\caption{First hitting times of an Ornstein Uhlenbeck process through two constant boundaries $b=-a=1$: i) Joint pdf ii) Contour lines of the joint pdf iii) Density of the copula iv) Contour lines of the density of the copula.}
\label{Fig:OU}       
\end{figure}
\begin{figure}[htp]
\centering
\includegraphics[height=10cm]{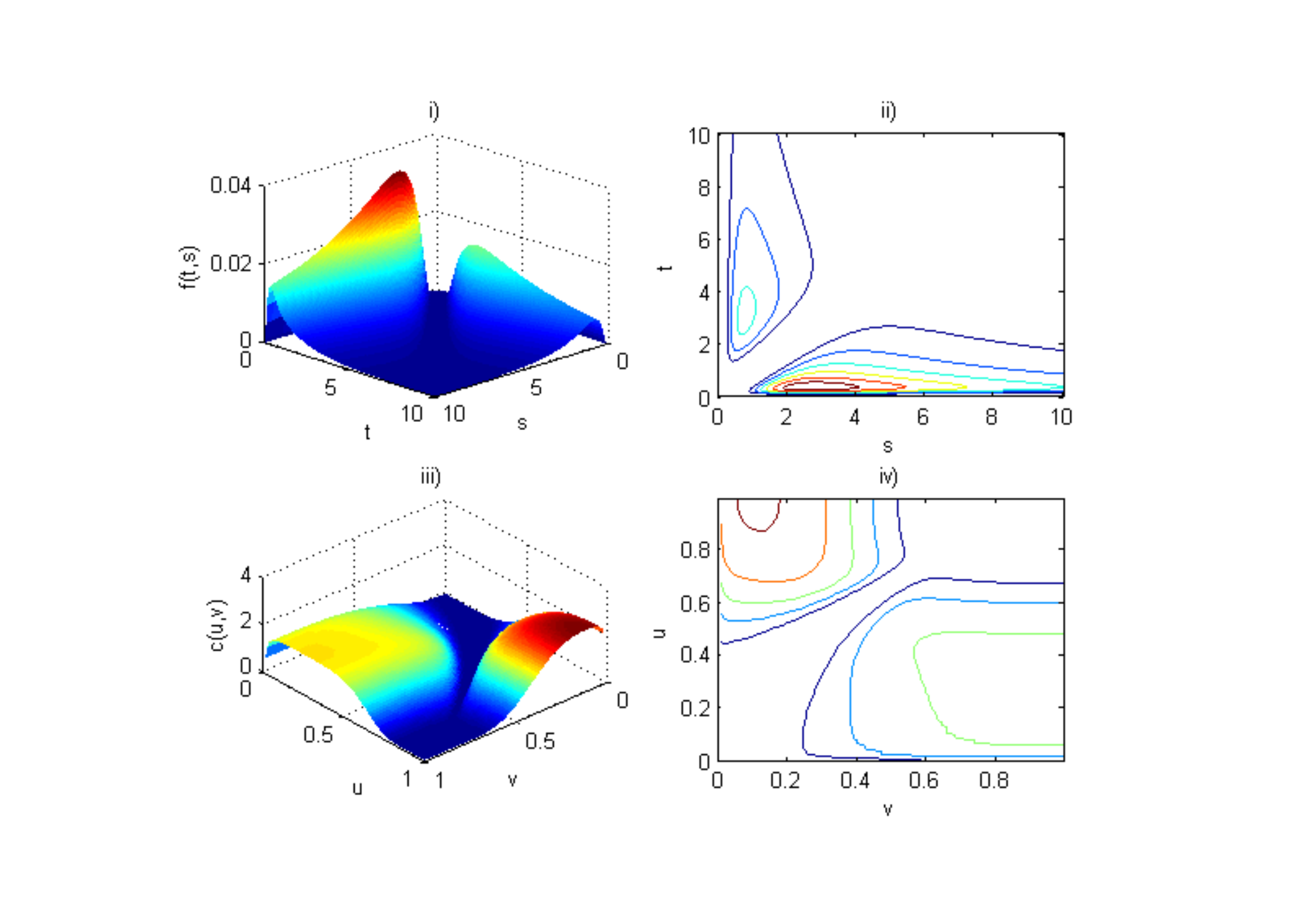}
\caption{First hitting times of an Ornstein Uhlenbeck process through two constant boundaries $a=-1$ and $b=1.5$: i) Joint pdf ii) Contour lines of the joint pdf iii) Density of the copula iv) Contour lines of the density of the copula.}
\label{Fig:OU2}       
\end{figure}
 
The Laplace transforms $g_{a}^{\lambda}(x_0)$ and $g_{b}^{\lambda}(x_0)$ for the OU process can be found in \cite{Borod}. However the presence of the parabolic cylinder functions in their expression discourage their numerical inversion. 
 
\acks We are grateful to an anonymous referee for his interesting and constructive comments to improve the paper.
Work supported in part by University of Torino Grant 2012 \lq \lq Stochastic Processes and their Applications\rq\rq and by project A.M.A.L.F.I. - Advanced Methodologies for the AnaLysis and management of the
Future Internet (Universit\`{a} di Torino/Compagnia di San
Paolo).

%

%

\end{document}